\documentclass[oneside,english]{amsart}
\usepackage[T1]{fontenc}
\usepackage[latin9]{inputenc}
\usepackage{textcomp}
\usepackage{amsbsy}
\usepackage{amstext}
\usepackage{amsthm}
\usepackage{amssymb}

\makeatletter
\numberwithin{equation}{section}
\numberwithin{figure}{section}
\theoremstyle{plain}
\newtheorem{thm}{\protect\theoremname}
\theoremstyle{plain}
\newtheorem{cor}[thm]{\protect\corollaryname}
\theoremstyle{definition}
\newtheorem{defn}[thm]{\protect\definitionname}
\theoremstyle{remark}
\newtheorem{rem}[thm]{\protect\remarkname}
\theoremstyle{plain}
\newtheorem{lem}[thm]{\protect\lemmaname}
\theoremstyle{plain}
\newtheorem{prop}[thm]{\protect\propositionname}

\usepackage[a4paper, margin=3cm]{geometry}

\theoremstyle{definition}
\newtheorem*{assumption}{Assumption}

\theoremstyle{theorem}
\newtheorem*{chen}{Generalized Chen's Conjecture}

\DeclareMathOperator{\tr}{tr}

\DeclareMathOperator{\Diff}{Diff}
\DeclareMathOperator{\phg}{phg}

\usepackage{graphicx}
\usepackage{amsmath}
\usepackage{tikz}
\usepackage{tikz-3dplot}
\usetikzlibrary{math,calc,arrows}

\makeatother

\usepackage{babel}
\providecommand{\corollaryname}{Corollary}
\providecommand{\definitionname}{Definition}
\providecommand{\lemmaname}{Lemma}
\providecommand{\propositionname}{Proposition}
\providecommand{\remarkname}{Remark}
\providecommand{\theoremname}{Theorem}

\begin{document}
\title{Biharmonic maps on conformally compact manifolds}
\author{Marco Usula}
\begin{abstract}
We study biharmonic maps between conformally compact manifolds, a
large class of complete manifolds with bounded geometry, asymptotically
negative curvature, and smooth compactification. These metrics provide
a far-reaching generalization of hyperbolic space. We work on the
class of simple $b$-maps, i.e. maps which send interior to interior,
boundary to boundary, and are transversal to the boundary of the target
manifold. The main result of this paper is a non-existence result:
if a simple $b$-map $u:\left(M,g\right)\to\left(N,h\right)$ between
conformally compact manifolds is biharmonic, its restriction to the
boundary is non-constant, and moreover $\left(N,h\right)$ is non-positively
curved, then $u$ is harmonic. We do not assume any integrability
condition on $u$: in particular, $u$ is not required to have finite
energy, nor is its tension field required to be in $L^{p}$ for any
$p$. Our result implies the following version of the Generalized
Chen's Conjecture: if $\left(N,h\right)$ is a non-positively curved
conformally compact manifold, and $\Sigma\hookrightarrow N$ is a
properly embedded submanifold with boundary meeting $\partial N$
transversely, then $\Sigma$ is biharmonic if and only if it is minimal.
\end{abstract}

\maketitle
\sloppy

\section{Introduction}

Let $\left(M,g\right)$ and $\left(N,h\right)$ be two Riemannian
manifolds. A map $u\in C^{\infty}\left(M,N\right)$ is said to be
\emph{harmonic} if it is a critical point of the energy functional
\[
\mathcal{E}_{1}\left(u\right)=\frac{1}{2}\int_{M}\left|du\right|^{2}\text{dVol}_{g}.
\]
If $M$ is non-compact, this energy can be infinite, but the definition
is still valid if one restricts to compactly supported variations
of $u$. Being a critical point of $\mathcal{E}_{1}$ is equivalent
to the Euler--Lagrange equations
\[
\tr_{g}\nabla du=0,
\]
where the differential $du$ is seen as a section of $T^{*}M\otimes u^{*}TN$,
and $\nabla$ is the connection on this bundle induced by the Levi-Civita
connection on $T^{*}M$ and the pull-back via $u$ of the Levi-Civita
connection on $TN$. The resulting vector field along $u$, $\tau\left(u\right):=\tr_{g}\nabla du$,
is called the \emph{tension field }of $u$. The equation $\tau\left(u\right)=0$
is a second order semilinear elliptic partial differential equation.

The study of harmonic maps was pioneered by Eells and Sampson in \cite{EellsSampson},
and it has been a central topic in geometric analysis ever since.
Harmonic maps are closely related to minimal immersions: if $u:M\to N$
is a Riemannian immersion (i.e. the pull-back via $u$ of the metric
on $N$ equals the metric on $M$), then $u$ is harmonic if and only
if it is minimal. When the curvature of the target manifold $N$ is
non-positive, there are very strong existence and uniqueness results
in a given homotopy class: it is proved in \cite{EellsSampson} that,
if $M$ is compact and the curvature of $N$ is non-positive, then
the gradient flow for the energy functional (the so-called harmonic
map heat flow) exists for all times and converges to a harmonic map.
Furthermore, as proved by Hartman in \cite{HartmanHomotopicHarmonic},
the limiting function is unique. This establishes a nonlinear analogue
of the Hodge Theorem when the curvature of $N$ is non-positive. The
literature on harmonic maps is vast, and we refer for example to \cite{JostBook}
and the references therein for a more complete picture.

\emph{Biharmonic maps} provide a natural generalization of harmonic
maps. They are critical points of the \emph{bienergy }functional
\[
\mathcal{E}_{2}\left(u\right)=\frac{1}{2}\int_{M}\left|\tau\left(u\right)\right|^{2}\text{dVol}_{g}.
\]
Analogously to harmonic maps, the Euler--Lagrange equation for this
functional is a fourth order semilinear elliptic PDE for $u$,
\[
J_{u}\tau\left(u\right)=0.
\]
Here $J_{u}\in\Diff^{2}\left(M;u^{*}TN\right)$ is a second order
Laplace-type elliptic operator on sections of $u^{*}TN$, defined
as
\[
J_{u}:=\Delta+\tr_{g}R_{h}\left(du,\cdot\right)du.
\]
The operator $\Delta$ is the (positive) rough Laplacian associated
to the pull-back of the Levi-Civita connection on $u^{*}TN$, while
$R_{h}$ is the Riemann tensor of $h$. We will refer to $J_{u}$
as the \emph{Jacobi operator }associated to $u$.

Biharmonic maps were introduced by Eells and Lemaire in \cite{EellsLemaireSelectedTopics},
and there has been a steadily growing interest in this generalization
of harmonic maps since then. Since harmonic maps are absolute minimizers
for the bienergy functional, it is natural to look for biharmonic
maps which are not harmonic; maps of this type are called \emph{proper
biharmonic}. There is a rich literature on this subject, and we refer
to \cite{MontaldoOniciucSurvey, YeLinBangYenChen} for a general survey.
There are many examples of proper biharmonic maps with target manifolds
of non-negative curvature: we mention for example \cite{BalmusMontaldoOniciucSpheres, CaddeoMontaldoOniciucSpheres, SiffertNewBiharmonic, SiffertExplicitProperPharmonic},
and some intriguing examples that seem special of dimension 4 \cite{MontaldoOniciucRattoRotationally}.
Many of these works provide examples of biharmonic submanifolds, i.e.
Riemannian submanifolds whose inclusion is biharmonic.

By contrast, the case of non-positively curved targets is much more
restrictive. When $M$ is \emph{closed }(i.e. compact without boundary)
and $N$ is non-positively curved, a direct application of the Weitzenböck
formula implies that all biharmonic maps must be harmonic \cite{JiangI, JiangII}.
This result has been generalized to the case of complete domains,
under various finiteness conditions. In particular, in \cite{NUG},
it is proved that biharmonic maps $u:\left(M,g\right)\to\left(N,h\right)$
with $\left(M,g\right)$ complete, $\left(N,h\right)$ non-positively
curved, and \emph{with} \emph{both finite energy and bienergy}, must
necessarily be harmonic. In \cite{Maeta}, the assumptions of finiteness
of the bienergy is relaxed, and the same result is proved under the
hypothesis $\tau\left(u\right)\in L^{p}$ for some $p\in\left(1,+\infty\right)$.

In \cite{ChenI}, B. Y. Chen conjectured that every biharmonic submanifold
of $\mathbb{R}^{n}$ must be minimal. Meanwhile, it has been proved
in \cite{CaddeoMontaldoOniciucSpheres} that every biharmonic surface
in $\text{H}^{3}$ must be minimal, and that every pseudo-umbilical
biharmonic submanifold $M^{m}\subset\text{H}^{n}$ with $m\not=4$
must be minimal. This led to the formulation of the following generalized
version of Chen's conjecture \cite{CaddeoMontaldoPiuOnBiharmonicMaps}:

\begin{chen}Any biharmonic submanifold of a non-positively curved
manifold $\left(N,h\right)$ must be minimal.\end{chen}

This conjecture is now known to be false in full generality, with
counterexamples provided in \cite{YeLinLiangChenConjecture} on certain
negatively curved manifolds.

In this paper, we focus on biharmonic \emph{simple $b$-maps between
conformally compact manifolds}. A simple $b$-map is just a smooth
map $u:M\to N$ which satisfies two properties:
\begin{enumerate}
\item $u$ sends the interior $M^{\circ}$ to the interior $N^{\circ}$,
and the boundary $\partial M$ to the boundary $\partial N$;
\item $u$ is transversal to $\partial N$.
\end{enumerate}
The class of \emph{conformally compact metrics} on a compact manifold
with boundary $M$ is a large class of complete metrics on the interior
$M^{\circ}$, which develop a ``double pole'' at $\partial M$.
More precisely, given a boundary defining function $x$ for $M$ (cf.
Definition \ref{def:bdf}), a conformally compact metric $g$ on $M$
is a metric on the interior $M^{\circ}$ such that the conformally
related metric $x^{2}g$ extends to a metric on the whole of $M$.
The prototypical example is the hyperbolic metric on the unit ball
$B^{n}$ in $\mathbb{R}^{n}$,
\[
g_{\text{H}^{n}}=\frac{4d\boldsymbol{y}^{2}}{\left(1-\left|\boldsymbol{y}\right|^{2}\right)^{2}}
\]
where $\boldsymbol{y}=\left(y^{1},...,y^{n}\right)\in\mathbb{R}^{n}$.
In this case, the boundary defining function is $1-\left|\boldsymbol{y}\right|^{2}$.
Just as for hyperbolic space, conformally compact metrics are complete
with bounded geometry; any point of the boundary is at infinite distance
from any point of the interior; the boundary inherits from the metric
a conformal class called the \emph{conformal infinity} of the metric;
finally, the sectional curvatures are all asymptotically negative.
We will recall some properties of conformally compact metrics in §\ref{subsec:Conformally-compact-manifolds}.

Conformally compact geometry has been a central topic of study in
geometric analysis since the 80s, starting with the influential work
of Fefferman and Graham: in \cite{FeffermanGraham}, they studied
expansions of conformally compact \emph{Einstein }metrics, in order
to obtain invariants of their conformal infinity. By now, conformally
compact geometry is widely regarded as the correct framework to understand
the Riemannian version of the AdS-CFT correspondence in physics \cite{BiquardAdS}.
From a more analytic side, the study of conformally compact manifolds
and their natural geometric elliptic operators was pioneered by Mazzeo
\cite{MazzeoPhD, MazzeoHodge}, and Mazzeo--Melrose \cite{MazzeoMelroseResolvent}.

The main result of this paper is the following non-existence result
for proper biharmonic simple $b$-maps:
\begin{thm}
Let $\left(M^{m+1},g\right)$ and $\left(N^{n+1},h\right)$ be conformally
compact manifolds, with $\left(N,h\right)$ of non-positive sectional
curvature. Let $u:M\to N$ be a simple $b$-map. If $m\geq1$, assume
that $u_{|\partial M}$ is not constant on any connected component
of $\partial M$. Then $u$ is biharmonic if and only if it is harmonic.
\end{thm}

The condition on the restriction of $u_{|\partial M}$ is necessary:
we show this by exhibiting an example of a proper biharmonic map from
$\text{H}^{m+1}$ with $m>0$, to $\text{H}^{n+1}$, which does not
satisfy this assumption (cf. Proposition \ref{prop:example-proper-biharmonic}).
We remark that, unlike \cite{NUG, Maeta}, \emph{we do not assume
any finiteness property on the energy or the bienergy of $u$}. As
we shall see, for any simple $b$-map $u:\left(M,g\right)\to\left(N,h\right)$,
the pointwise energy $\left|du\right|^{2}$ and the pointwise bienergy
$\left|\tau\left(u\right)\right|^{2}$ are both naturally bounded
functions on $M$. However, we do not assume these functions to be
integrable; in fact, for a \emph{generic} simple $b$-map, the functions
$\left|du\right|^{2}$ and $\left|\tau\left(u\right)\right|^{2}$
do not vanish at infinity, and therefore they are generically not
in $L^{p}$ for any $p\in\left(1,+\infty\right)$, since the volume
of $\left(M,g\right)$ is infinite. As an immediate corollary, since
harmonic immersions are minimal, we obtain the following version of
the Generalized Chen's Conjecture:
\begin{cor}
Let $\left(N,h\right)$ be a conformally compact manifold of non-positive
sectional curvature, and let $M$ be a compact manifold with boundary.
Let $u:M\to N$ be an immersion mapping $\partial M$ to $\partial N$,
$M^{\circ}$ to $N^{\circ}$, and transversal to $\partial N$. Then
$u$ is biharmonic if and only if it is minimal.
\end{cor}

Let us comment on the techniques used in this work. As we shall see,
the class of simple $b$-maps is naturally adapted to conformally
compact geometry, and makes the problem discussed in this paper amenable
to the analytic techniques developed by Mazzeo and Melrose. In general,
both the harmonic and biharmonic map equations are semilinear elliptic
PDEs: this means that the linearizations (at a given map $v$) of
the tension field map $u\mapsto\tau\left(u\right)$ and the bitension
field map $u\mapsto\tau_{2}\left(u\right):=J_{u}\tau\left(u\right)$,
are both elliptic differential operators on sections of $v^{*}TN$.
In the conformally compact context, \emph{a priori} these differential
operators make sense only in the interior, because the conformally
compact metrics (and their Levi-Civita connections) do not extend
to the boundary in a straigthforward way. However, we shall see in
§\ref{subsec:Harmonic-and-biharmonic simple b-maps} that these operators
do extend canonically over the whole of $M$ to \emph{$0$-elliptic
$0$-differential operators}. The class of $0$-differential operators
was introduced by Mazzeo and Melrose, and contains all the geometric
operators associated to a conformally compact metric. There is a very
refined elliptic theory for these operators \cite{MazzeoMelroseResolvent, MazzeoPhD, MazzeoHodge, Hintz0calculus, LeeFredholm, MazzeoEdgeI},
to which the author has contributed in \cite{UsulaPhD, Usula0BVP}.
Although we will use this elliptic theory for the proof of the main
theorem, this paper is self-contained: we will recall the main definitions
and the parts of the theory that we need in §\ref{sec:0-world}.

\subsection*{Acknowledgements}

The author wishes to thank Stefano Montaldo, for his encouragement
and for various fruitful discussions on biharmonic maps. This work
was supported by the EoS grant 40007524.

\section{\label{sec:0-world}The $0$-world}

\subsection{\label{subsec:A-remark-on-phg}A preliminary remark on polyhomogeneity}

A compact manifold with boundary $M$ comes equipped with various
natural rings of functions which generalize the ring $C^{\infty}\left(M\right)$.
One of them is the ring of \emph{bounded polyhomogeneous functions}.
Roughly speaking, polyhomogeneous functions on $M$ are smooth functions
on the interior $M^{\circ}$ which admit a Taylor-like expansion near
the boundary; in contrast with the smooth case, however, the terms
in the expansion are allowed to be of the form $x^{\alpha}\left(\log x\right)^{l}$,
where $x$ is a boundary defining function for $M$, and $\left(\alpha,l\right)\in\mathbb{C}\times\mathbb{N}$.
The precise description of this notion is summarized below, and we
refer to Chapter 4 of \cite{MelroseCorners} for further details.
\begin{defn}
\label{def:bdf}A \emph{boundary defining function }for $M$ is a
smooth function $x:M\to[0,+\infty)$ such that $x^{-1}\left(0\right)=\partial M$
and $dx$ is nowhere vanishing along $\partial M$.
\end{defn}

\begin{defn}
An \emph{index set} is a subset of $\mathbb{C}\times\mathbb{N}$ such
that if $\left(\alpha,l\right)\in\mathcal{E}$, then $\left(\alpha+k,l'\right)\in\mathcal{E}$
for every $k\in\mathbb{N}$ and $l'\leq l$, and for every $\delta\in\mathbb{R}$
the set $\mathcal{E}_{\delta}=\left\{ \left(\alpha,l\right)\in\mathcal{E}:\Re\left(\alpha\right)\leq\delta\right\} $
is finite.
\end{defn}

\begin{defn}
A\emph{ polyhomogeneous function on $M$ with index set $\mathcal{E}$
}is a smooth function $u$ on $M^{\circ}$ which admits an asymptotic
expansion at the boundary, of the form
\[
u\sim\sum_{\left(\alpha,l\right)\in\mathcal{E}}u_{\alpha,l}x^{\alpha}\left(\log x\right)^{l},
\]
where $x$ is a boundary defining function for $M$, and $u_{\alpha,l}\in C^{\infty}\left(M\right)$.
We denote by $\mathcal{A}_{\phg}^{\mathcal{E}}\left(M\right)$ the
space of polyhomogeneous functions on $M$ with index set $\mathcal{E}$.
\end{defn}

\begin{rem}
The meaning of the asymptotic expansion above is that, for every $\delta\in\mathbb{R}$,
the truncated difference
\[
S_{\delta}:=u-\sum_{\begin{smallmatrix}\left(\alpha,l\right)\in\mathcal{E}\\
\Re\left(\alpha\right)\leq\delta
\end{smallmatrix}}u_{\alpha,l}x^{\alpha}\left(\log x\right)^{l}
\]
is in $x^{\delta}L^{\infty}\left(M\right)$, as well as all its multiple
derivatives by vector fields tangential to the boundary. In the language
developed in \cite{MelroseCorners}, we say that the truncated difference
$S_{\delta}$ is $O\left(x^{\delta}\right)$ \emph{conormal}. Note
that, by Taylor's Theorem, $\mathcal{A}_{\phg}^{\mathbb{N}}\left(M\right)=C^{\infty}\left(M\right)$.
\end{rem}

\begin{rem}
Although functions in $\mathcal{A}_{\phg}^{\mathcal{E}}\left(M\right)$
are allowed to blow up at the boundary, and are therefore defined
a priori only on $M^{\circ}$, the spaces $\mathcal{A}_{\phg}^{\mathcal{E}}\left(M\right)$
are intrinsically associated to the compact manifold: indeed, elements
of $\mathcal{A}_{\phg}^{\mathcal{E}}\left(M\right)$ can be uniquely
interpreted as extendible distributions on $M$ (cf. \cite{MelroseCorners}).
\end{rem}

\begin{assumption}In this paper, when no index set is specified,
every polyhomogeneous function is assumed to be real-valued, bounded,
and with index set $\mathcal{E}\subseteq\mathbb{R}\times\mathbb{N}$.
In particular, \emph{no negative powers of $x$, complex powers of
$x$, or leading log terms, are allowed}. Similarly, all polyhomogeneous
objects (sections of vector bundles, connections, maps, etc.) are
assumed to be locally described in terms of real-valued bounded polyhomogeneous
functions in the sense specified here.\end{assumption}

While polyhomogeneous functions might seem a bit exotic at first sight,
working with these objects in our context is both natural and necessary.
As already mentioned, in our formulation of the problem, the harmonic
and biharmonic map equations are \emph{semilinear $0$-elliptic equations
}(cf. §\ref{subsec:0-differential-operators}); solutions of equations
of this type, with smooth ``boundary values'', are always conormal
and often polyhomogeneous, but they are very rarely smooth. We refer
to \cite{BoundaryRegularityPE, MazzeoYamabeRegularity} for two important
examples of non-linear problems of this type, which can help elucidate
this phenomenon.

\subsection{\label{subsec:Conformally-compact-manifolds}Conformally compact
manifolds}

Let $M^{m+1}$ be a compact manifold with boundary. Denote by $M^{\circ}$
the interior of $M$, and by $\partial M$ its boundary.
\begin{defn}
A \emph{conformally compact metric} (or \emph{$0$-metric}) on $M$
is a smooth metric $g$ on $M^{\circ}$ such that, for some (hence
every) boundary defining function $x$ on $M$, the conformal rescaling
$x^{2}g$ extends to a polyhomogeneous metric on the whole of $M$.
\end{defn}

Conformally compact metrics have been studied intensely, particularly
in relation to the Einstein equation and the AdS/CFT correspondence
\cite{BiquardAdS}. Although, as metrics, they are defined only in
the interior, they can be seen as bundle metrics on a natural vector
bundle over the whole of $M$. Denote by $\mathcal{V}_{0}\left(M\right)$
the Lie algebra of smooth vector fields on $M$ vanishing along $\partial M$.
Its elements are called \emph{$0$-vector fields}. $0$-vector fields
are just ordinary vector fields locally near a point of $M^{\circ}$,
but near a point $p\in\partial M$ they can be written in half-space
coordinates $\left(x,\boldsymbol{y}\right)$ centered at $p$ as linear
combinations of
\[
x\partial_{x},x\partial_{y^{1}},...,x\partial_{y^{m}}.
\]
The space $\mathcal{V}_{0}\left(M\right)$ is a locally finitely generated
projective $C^{\infty}\left(M\right)$ module, and therefore by the
Serre--Swan Theorem it can be realized as the module of sections
of a unique smooth vector bundle $^{0}TM$, the \emph{$0$-tangent
bundle}. The inclusion $\mathcal{V}_{0}\left(M\right)\to\mathcal{V}\left(M\right)$
is a morphism of $C^{\infty}\left(M\right)$ modules, and therefore
it induces a bundle map $\#:{}^{0}TM\to TM$ called the \emph{anchor
map}. Since $0$-vector fields span $TM$ pointwise at interior points,
the anchor map identifies $^{0}TM_{|M^{\circ}}$ with $TM^{\circ}$
canonically. However, the rank of $\#$ drops to $0$ along $\partial M$.
Along with $^{0}TM$, we can define all the usual tensor bundles,
and in particular the exterior powers $^{0}\Lambda^{k}:=\Lambda^{k}\left(^{0}T^{*}M\right)$
and the symmetric powers $^{0}S^{k}:=S^{k}\left(^{0}T^{*}M\right)$.

The $0$-tensor bundles introduced here are very useful, because they
allow us to interpret certain singular/degenerate tensors on $M^{\circ}$
as smooth or polyhomogeneous sections of appropriate vector bundles
on $M$. For example, from the point of view of $M^{\circ}$, conformally
compact metrics develop a ``double pole'' near $\partial M$; however,
it is easy to see that conformally compact metrics are precisely bundle
metrics on $^{0}TM$, i.e. positive definite sections of $S^{2}\left(^{0}T^{*}M\right)$.
Similarly, the smooth sections of $^{0}\Lambda^{k}$ (called $0$-$k$-forms)
on $M$ are precisely the smooth $k$-forms $\omega$ on $M^{\circ}$
such that, for some (and hence every) smooth boundary defining function
$x$, $x^{k}\omega$ extends smoothly to a $k$-form over $M$. Although
$0$-$k$-forms are singular seen as standard $k$-forms (they develop
a pole of order $k$ near $\partial M$), they are smooth sections
of $^{0}\Lambda^{k}$ on the whole of $M$.

Let's list some of the most important properties of a conformally
compact metric $g$ on $M$. These results have been proved in \cite{MazzeoPhD},
to which we refer for more details:
\begin{enumerate}
\item $g$ is complete with bounded geometry;
\item $g$ induces a conformal class on $\partial M$, defined as $\mathfrak{c}_{\infty}\left(g\right):=\left[x^{2}g_{|\partial M}\right]$
where $x$ ranges among the boundary defining functions on $M$; $\mathfrak{c}_{\infty}\left(g\right)$
is called the \emph{conformal infinity} of $g$;
\item $g$ has asymptotically negative curvature: more precisely, if $\left\{ p_{k}\right\} _{k\in\mathbb{N}}$
is a sequence in $M^{\circ}$ such that $p_{k}\to p\in\partial M$,
and $\pi_{k}$ is a $2$-plane tangent to $M$ at $p_{k}$, then the
sequence of sectional curvatures $\kappa_{g}\left(\pi_{k}\right)$
satisfies
\[
\lim_{k\to\infty}\kappa_{g}\left(\pi_{k}\right)=-\left|\frac{dx}{x}\right|_{g|p}^{2}
\]
for some boundary defining function $x$ on $M$.
\end{enumerate}
Note that the last property is well-posed: indeed, if $x$ and $\tilde{x}$
are two boundary defining functions on $M$, then the two $0$-$1$-forms
$dx/x$ and $d\tilde{x}/\tilde{x}$ agree on the boundary. In other
words, $\left(dx/x\right)_{|\partial M}$ is a \emph{canonical $0$-$1$-form
on $M$ defined along $\partial M$}, and therefore its pointwise
squared norm $\left|\left(dx/x\right)_{|\partial M}\right|_{g}^{2}$
is an invariant of $g$. $g$ is fittingly called \emph{asymptotically
hyperbolic }if this function equals $1$ identically. The paradigmatic
example of conformally compact metric is the hyperbolic metric itself:
indeed, we can see the hyperbolic metric in dimension $n$ as the
metric
\[
\frac{4d\boldsymbol{y}^{2}}{\left(1-\left|\boldsymbol{y}\right|^{2}\right)^{2}}
\]
in the interior of the closed unit ball in $\mathbb{R}^{n}$, with
$\boldsymbol{y}=\left(y^{1},...,y^{n}\right)$.

Conformally compact geometry is strongly tied to hyperbolic geometry.
In particular, if $g$ is a polyhomogeneous conformally compact metric
on $M$, then for every $p\in\partial M$ there is a ``model hyperbolic
metric'' $\left(M_{p},g_{p}\right)$ which approximates $\left(M,g\right)$
near $p$. More precisely, denote by $M_{p}$ the inward-pointing
closed half of $T_{p}M$. Choose a local boundary defining function
$x$ for $M$ defined near $p$, and call $h_{p}$ the restriction
$\left(x^{2}g\right)_{p}$ to $T_{p}M$. The covector $dx_{p}$, seen
as a linear map $T_{p}M\to\mathbb{R}$, restricts to a smooth map
$M_{p}\to[0,+\infty)$ vanishing simply along $\partial M_{p}=T_{p}\partial M$;
in other words, $\tilde{x}:=dx_{p}$ is a boundary defining function
for $M_{p}$. We define
\[
g_{p}=\frac{d\tilde{x}^{2}+h_{p}}{\tilde{x}^{2}}.
\]
This metric is only defined in $M_{p}^{\circ}$, but it is conformally
related to the smooth metric $h_{p}$. It is immediate to check that
$g_{p}$ is well-defined independently of the choice of $x$.

The metric $g_{p}$ has constant negative sectional curvature equal
to $-\left|dx/x\right|_{g|p}^{2}$; in other words, $\left(M_{p},g_{p}\right)$
is identified with the half-space model of the rescaled hyperbolic
space, up to isometries that fix the origin. To get a precise identification
between $\left(M_{p},g_{p}\right)$ and the hyperbolic half-space
metric, define $a_{p}=\left|dx/x\right|_{g|p}$, and choose a local
boundary defining function $x$ for $M$ defined near $p$. The conformal
rescaling $x^{2}g$ induces a metric $h_{0}$ on $\partial M$ defined
near $p$. Choose normal coordinates $\boldsymbol{y}=\left(y^{1},...,y^{m}\right)$
centered at $p$ for the rescaled metric $a_{p}^{-2}h_{0}$. Then,
in the coordinates $\left(x,\boldsymbol{y}\right)$, the metric $g$
takes the form
\begin{align*}
g & =a_{p}^{-2}\frac{dx^{2}+d\boldsymbol{y}^{2}}{x^{2}}+O\left(\left(x^{2}+\left|\boldsymbol{y}\right|^{2}\right)^{\varepsilon}\right)
\end{align*}
for some $\varepsilon>0$. With slight abuse of notation, let us use
$\left(x,\boldsymbol{y}\right)$ to denote also the induced global
linear coordinates on $T_{p}M$. Then, in these coordinates, we have
\[
g_{p}=a_{p}^{-2}\frac{dx^{2}+d\boldsymbol{y}^{2}}{x^{2}}.
\]

\begin{defn}
A tuple $\left(x,\boldsymbol{y}\right)$ of half-space coordinates
for $M$ centered at $p$, with respect to which $g_{p}$ takes the
form above, are called \emph{normal half-space coordinates for $\left(M,g\right)$
at $p$.}
\end{defn}

\subsection{\label{subsec:0-connections}$0$-connections}
\begin{defn}
Let $E\to M$ be a vector bundle on $M$. A smooth \emph{$0$-connection
}on $E$ is a scalar-bilinear map
\begin{align*}
\nabla:\mathcal{V}_{0}\left(M\right)\times C^{\infty}\left(M;E\right) & \to C^{\infty}\left(M;E\right)\\
\left(V,u\right) & \mapsto\nabla_{V}u
\end{align*}
which is $C^{\infty}\left(M\right)$ linear in the first argument
and satisfies the Leibniz identity in the second argument.
\end{defn}

$0$-connections are the ``natural'' connections associated to the
Lie algebra $\mathcal{V}_{0}\left(M\right)$. An equivalent way to
characterize a $0$-connection is in terms of its connection $1$-forms:
if $D$ is a genuine connection on $E$, then a $0$-connection $\nabla$
is a connection in the interior such that the difference $\mathfrak{gl}\left(E\right)$-valued
$1$-form $D-\nabla$ extends from $M^{\circ}$ to a smooth $\mathfrak{gl}\left(E\right)$-valued
$0$-$1$-form on $M$. We can extend the definition above to \emph{polyhomogeneous
}$0$-connections, by asking that the difference $D-\nabla$ (for
an arbitrary genuine smooth connection $D$) to be a polyhomogeneous
$\mathfrak{gl}\left(E\right)$-valued $0$-$1$-form on $M$. This
definition does not depend on $D$, because if $D'$ is another smooth
connection, then $D-D'$ is a smooth section of $\Lambda^{1}\otimes\mathfrak{gl}\left(E\right)$,
or equivalently a smooth section of $^{0}\Lambda^{1}\otimes\mathfrak{gl}\left(E\right)$
vanishing along the boundary.

It is easy to see that all the Levi-Civita connections on tensor bundles
of $M^{\circ}$ associated to a conformally compact metric, extend
to $0$-connections on the corresponding $0$-tensor bundle. In particular,
the Levi-Civita connection on $TM^{\circ}$ extends to a $0$-connection
on $^{0}TM$. If $\nabla$ is a $0$-connection, the curvature $2$-form
extends from the interior to a $\mathfrak{gl}\left(E\right)$-valued
$0$-$2$-form. In particular, if $g$ is a conformally compact metric
on $M$, its Riemann tensor $R_{g}$ extends from the interior to
a bundle map
\[
R_{g}:{^{0}TM}\otimes{^{0}TM}\otimes{^{0}TM}\to{^{0}TM}
\]
which satisfies the usual symmetries. The author studied $0$-connections
in his PhD thesis \cite{UsulaPhD} and in the paper \cite{UsulaYM},
particularly in relation to the Yang--Mills and the self-duality
equations.

\subsection{\label{subsec:0-differential-operators}$0$-differential operators}

The $0$-connections considered in §\ref{subsec:0-connections} are
particular examples of a natural class of differential operators on
$M$ associated to the Lie algebra $\mathcal{V}_{0}\left(M\right)$.
\begin{defn}
An operator $L\in\Diff^{k}\left(M\right)$ is said to be a \emph{$0$-differential
operator with polyhomogeneous coefficients} if it can be written as
a linear combination of compositions of smooth $0$-vector fields,
with coefficients in the ring of polyhomogeneous functions. Equivalently,
$L$ is a $0$-differential operator of order $\leq k$ if and only
if, in half-space coordinates $\left(x,\boldsymbol{y}\right)$ centered
at a point $p\in\partial M$, it can be written as
\[
\sum_{j+\left|\alpha\right|\leq k}L_{j,\alpha}\left(x,\boldsymbol{y}\right)\left(x\partial_{x}\right)^{j}\left(x\partial_{\boldsymbol{y}}\right)^{\alpha}
\]
where $L_{j,\alpha}\left(x,\boldsymbol{y}\right)$ are polyhomogeneous
functions locally defined near the origin. We denote by $\Diff_{0}^{k}\left(M\right)$
the space of $0$-differential operators of order $\leq k$. $0$-differential
operators between sections of vector bundles are defined similarly.
\end{defn}

$0$-differential operators were introduced and studied by Mazzeo
in \cite{MazzeoPhD} and Mazzeo--Melrose in \cite{MazzeoMelroseResolvent},
and are by now very well-studied objects. Every natural geometric
operator associated to a conformally compact metric $g$ on $M$ (Laplacians,
Dirac operators etc.) is a $0$-differential operator. For example,
the Levi-Civita $0$-connection on $^{0}TM$ is a first-order $0$-differential
operator from sections of $^{0}TM$ to sections of ${^{0}T^{*}M}\otimes{^{0}TM}$.
By construction, $0$-differential operators can never be elliptic,
due to the degeneracy of $0$-vector fields along $\partial M$. However,
if $L\in\Diff_{0}^{k}\left(M\right)$, then the principal symbol $\sigma\left(L\right)$
seen as a fibrewise homogeneous smooth function on $T^{*}M^{\circ}$,
extends to a function $^{0}\sigma\left(L\right)$ on $^{0}T^{*}M$
fibrewise homogeneous of degree $k$.
\begin{defn}
We call $^{0}\sigma\left(L\right)$ the \emph{principal $0$-symbol}
of $L$. $L$ is said to be \emph{$0$-elliptic} if $^{0}\sigma\left(L\right)\left(\xi\right)$
is invertible for every non-null $0$-covector $\xi$.
\end{defn}

The elliptic theory of $0$-differential operators was pioneered by
Mazzeo and Mazzeo--Melrose, and is by now very refined. The author
has worked extensively on this topic, particularly in the context
of boundary value problems \cite{Usula0BVP}. Let's summarize the
basic concepts, and recall the fundamental theorem of Mazzeo which
we will use in this paper.

The manifold $M$ comes equipped with a natural class of ``weighted
$0$-Sobolev spaces'' naturally adapted to the Lie algebra $\mathcal{V}_{0}\left(M\right)$.
Fix an auxiliary conformally compact metric $g$ on $M$. We denote
by $L_{0}^{2}\left(M\right)$ the usual space of $L^{2}$ functions
associated to $g$; as the notation suggests, $L_{0}^{2}\left(M\right)$
does not depend on $g$ but only on the smooth structure of $M$.
The $L^{2}$-based $0$-Sobolev spaces are
\[
H_{0}^{k}\left(M\right)=\left\{ u\in L_{0}^{2}\left(M\right):V_{1}\cdots V_{k}u\in L_{0}^{2},\forall V_{i}\in\mathcal{V}_{0}\left(M\right)\right\} .
\]
Again, the spaces $H_{0}^{k}\left(M\right)$ can be seen as the usual
Sobolev spaces associated to $g$, but their Banach topologies are
in fact independent of $g$. Finally, given $\delta\in\mathbb{R}$
and an auxiliary boundary defining function $x$, we define the weighted
$0$-Sobolev spaces
\[
x^{\delta}H_{0}^{k}\left(M\right)=\left\{ x^{\delta}u:u\in H_{0}^{k}\left(M\right)\right\} .
\]
Again, the Banach topology on this space is independent of $x$.

It is easy to check that, if $L\in\Diff_{0}^{k}\left(M\right)$, then
$L$ induces a bounded linear map
\[
L:x^{\delta}H_{0}^{k}\left(M\right)\to x^{\delta}L_{0}^{2}\left(M\right).
\]
On a \emph{closed }manifolds, elliptic operators are Fredholm between
Sobolev spaces of the correct orders. This is not true anymore in
the $0$-elliptic context. More precisely, if $L$ is $0$-elliptic,
Fredholmness of $L$ is conditional on the weight $\delta\in\mathbb{R}$,
and depends fundamentally on another microlocal model for $L$ near
every point of the boundary.
\begin{defn}
Let $L\in\Diff_{0}^{k}\left(M\right)$, and let $p\in\partial M$.
The \emph{normal operator} of $L$ at $p$ is an operator $N_{p}\left(L\right)\in\Diff_{0}^{k}\left(M_{p}\right)$,
invariant under positive dilations and translations by elements of
$T_{p}\partial M$, defined as follows: if $\left(x,\boldsymbol{y}\right)$
are half-space coordinates for $M$ centered at $p$, with respect
to which we have
\[
L=\sum_{j+\left|\alpha\right|\leq k}L_{j,\alpha}\left(x,\boldsymbol{y}\right)\left(x\partial_{x}\right)^{j}\left(x\partial_{\boldsymbol{y}}\right)^{\alpha},
\]
then the normal operator in the induced linear coordinates on $M_{p}$
is
\[
N_{p}\left(L\right)=\sum_{j+\left|\alpha\right|\leq k}L_{j,\alpha}\left(0,\boldsymbol{0}\right)\left(x\partial_{x}\right)^{j}\left(x\partial_{\boldsymbol{y}}\right)^{\alpha}.
\]
The \emph{indicial operator }of $L$ at $p$ is the dilation invariant
operator $I_{p}\left(L\right)\in\Diff_{0}^{k}\left(N_{p}^{+}\partial M\right)$
(here $N_{p}^{+}\partial M$ is the inward-pointing half of the fiber
at $p$ of the normal bundle of $\partial M$ in $M$) defined in
these coordinates as
\[
\hat{I}_{p}\left(L\right)=\sum_{j\leq k}L_{j,0}\left(0,\boldsymbol{0}\right)\left(x\partial_{x}\right)^{j}.
\]
The \emph{indicial polynomial }is the polynomial $I_{p}\left(L\right)$
obtained from $\hat{I}_{p}\left(L\right)$ by formally substituting
$x\partial_{x}$ with $z$. A number $\mu\in\mathbb{C}$ is called
an \emph{indicial root }of $L$ at $p$ if $I_{p}\left(L\right)\left(\mu\right)$
is not invertible.
\end{defn}

All the definitions above can be given an invariant description (cf.
\cite{MazzeoMelroseResolvent}), and extend easily to operators acting
between sections of vector bundles. Note that, a priori, the indicial
roots can vary with the point $p\in\partial M$. In practice, this
is not the case for many geometric applications, including the one
discussed in this paper. When the indicial roots are constant, they
determine a finite set of points of $\mathbb{C}$. If $L$ acts between
sections of a vector bundle equipped with a bundle metric, and it
is formally self-adjoint with respect to the volume form of a conformally
compact metric on $M^{m+1}$ (a nowhere vanishing section of $^{0}\Lambda^{m+1}$),
then its indicial roots are symmetric about the line $\Re\left(z\right)=\frac{m}{2}$
in $\mathbb{C}$.

We will use the following fundamental theorem of Mazzeo. This theorem
can be deduced\footnote{In \cite{MazzeoEdgeI}, only operators with smooth coefficients are
considered; in this paper, we need operators with bounded polyhomogeneous
coefficients, but the theory developed in \cite{MazzeoEdgeI} extends
easily to this case.} from Theorem 6.1 of \cite{MazzeoEdgeI}, but the reader can also
consult Theorem C of \cite{LeeFredholm} for a very similar result,
obtained by different methods.
\begin{thm}
\label{thm:(Mazzeo)}(Mazzeo) Let $\left(M^{m+1},g\right)$ be an
oriented conformally compact manifold, and let $E\to M$ be a vector
bundle equipped with a bundle metric. Let $J\in\Diff_{0}^{k}\left(M;E\right)$
be a formally self-adjoint $0$-elliptic $0$-differential operator
with polyhomogeneous coefficients. Assume that:
\begin{enumerate}
\item $J$ has constant indicial roots;
\item $J$ has no indicial roots with real part equal to $m/2$.
\end{enumerate}
Denote by $\left(a,b\right)$ the maximal interval containing $m/2$
which does not contain the real part of any indicial root of $J$.
If, for every $p\in\partial M$, the kernel of $N_{p}\left(J\right)$
in $L_{0}^{2}\left(M_{p};E_{p}\right)$ is trivial, then $J$ is Fredholm
of index zero as a map
\[
J:x^{\delta-\frac{m}{2}}H_{0}^{k}\left(M;E\right)\to x^{\delta-\frac{m}{2}}L_{0}^{2}\left(M;E\right)
\]
for every $\delta\in\left(a,b\right)$. Moreover, the $x^{\delta-\frac{m}{2}}L_{0}^{2}$
kernels of $J$ with $\delta\in\left(a,b\right)$ are all equal, and
are spanned by a finite number of polyhomogeneous functions with index
set $\mathcal{E}$ satisfying $\Re\left(\mathcal{E}\right)>b-\varepsilon$
for every $\varepsilon>0$.
\end{thm}

\section{Harmonic and biharmonic $b$-maps}

\subsection{\label{subsec:Simple-b-maps}Simple $b$-maps}

Let's fix two compact manifolds with boundary $M^{m+1}$ and $N^{n+1}$.
We say that a map $u:M\to N$ is \emph{polyhomogeneous} if its components
in smooth local coordinates are polyhomogeneous.
\begin{defn}
\label{def:simple-b-map}A \emph{simple $b$-map }is a polyhomogeneous
map $u:M\to N$ such that the pull-back of a polyhomogeneous boundary
defining function on $N$ is a polyhomogeneous boundary defining function
on $M$.
\end{defn}

\begin{rem}
Note that $u:M\to N$ is a simple $b$-map if and only if $u$ maps
$\partial M$ to $\partial N$ and $M^{\circ}$ to $N^{\circ}$, and
is transversal to $N$. Indeed, assume that $u:M\to N$ is a simple
$b$-map and let $\rho$ be a boundary defining function for $N$.
Then $u^{*}\rho$ is a boundary defining function for $M$, and therefore
$p\in\partial M$ if and only if $u^{*}\rho\left(p\right)=0$ if and
only if $\rho\left(u\left(p\right)\right)=0$ if and only if $u\left(p\right)\in\partial N$.
Moreover, let $p\in\partial M$ and let $v\in T_{p}M$ be inward-pointing
and transversal to the boundary. Then $d\left(u^{*}\rho\right)_{p}\left(v\right)>0$
because $u^{*}\rho$ is a boundary defining function for $M$. But
now $d\left(u^{*}\rho\right)_{p}=d\rho_{u\left(p\right)}\left(du_{p}\left(v\right)\right)>0$,
which means that $du_{p}\left(v\right)$ is inward-pointing and transversal
to the boundary of $N$. Therefore, $du_{p}\left(T_{p}M\right)+T_{u\left(p\right)}\partial N=T_{u\left(p\right)}N$.
Conversely, suppose that $u:M\to N$ maps boundary to boundary, interior
to interior, and is transversal to $N$. Let $\rho$ be a boundary
defining function for $N$. Then $u^{*}\rho\geq0$, $\left(u^{*}\rho\right)\left(p\right)=0$
if and only if $u\left(p\right)\in\partial N$ if and only if $p\in\partial M$.
Moreover, if $d\left(u^{*}\rho\right)_{p}=0$, then $d\rho_{u\left(p\right)}\circ du_{p}=0$.
Since $\rho$ is a boundary defining function for $N$, the kernel
of $d\rho_{u\left(p\right)}$ is precisely $T_{u\left(p\right)}\partial N$;
therefore, $d\rho_{u\left(p\right)}\circ du_{p}=0$ implies that $du_{p}$
has range in $T_{u\left(p\right)}\partial N$, contradicting the transversality
hypothesis.
\end{rem}

Fix a simple $b$-map $u:M\to N$.
\begin{lem}
\label{lem:pull-back-0-1-forms}Let $\omega$ be a polyhomogeneous
$0$-$1$-form on $N$. Then $u^{*}\omega$ is a polyhomogeneous $0$-$1$-form
on $M$.
\end{lem}

\begin{rem}
Here the notion of polyhomogeneity is the one specified in §\ref{subsec:A-remark-on-phg}:
in particular, polyhomogeneous $0$-$1$-forms have bounded pointwise
norms with respect to conformally compact metrics. 
\end{rem}

\begin{proof}
Let $X$ be a boundary defining function for $N$, and let $x=u^{*}X$.
Since $u$ is a simple $b$-map, $x$ is a boundary defining function
for $M$. Now, since $\omega$ is a $0$-$1$-form on $N$, in the
interior we have $\omega=X^{-1}\overline{\omega}$ where $\overline{\omega}$
is a $1$-form on $N$. Therefore, $u^{*}\omega=x^{-1}\overline{\eta}$
with $\overline{\eta}=u^{*}\overline{\omega}$. $\overline{\eta}$
is a $1$-form on $M$, so $x^{-1}\overline{\eta}$ is a $0$-$1$-form
on $M$ as claimed.
\end{proof}
\begin{cor}
\label{cor:du-is-0-object}The differential $du:TM^{\circ}\to TN^{\circ}$
extends by continuity to a bundle map $^{0}du:{^{0}TM}\to{^{0}TN}$
covering $u:M\to N$.
\end{cor}

\begin{proof}
The differential $du:TM^{\circ}\to TN^{\circ}$ can be identified
with a bundle map $du:TM^{\circ}\to u^{*}TN^{\circ}$ covering the
identity on $M$. Now, this map coincides with the dual of the bundle
map $u^{*}:u^{*}T^{*}N^{\circ}\to T^{*}M^{\circ}$ induced by the
pull-back on sections. By the previous lemma, $u^{*}$ extends to
a bundle map $^{0}u^{*}:u^{*}{^{0}T^{*}N}\to{^{0}T^{*}M}$, so taking
the dual we obtain a bundle map $^{0}du:{^{0}TM}\to u^{*}{^{0}TN}$,
which we interpret as a bundle map $^{0}du:{^{0}TM}\to{^{0}TN}$ covering
$u$.
\end{proof}
\begin{cor}
\label{cor:pull-back-LC-conn-is-0-conn}Let $A$ be a $0$-connection
on a vector bundle $E\to N$. Then the pull-back connection $u^{*}A$
is a $0$-connection on $u^{*}E$.
\end{cor}

\begin{proof}
Let $p\in M$, and let $s_{1},...,s_{k}$ be a local frame for $E$
defined near $u\left(p\right)$. By definition of $0$-connection,
$A$ is locally expressed in terms of this frame as a matrix of $0$-$1$-forms:
using the Einstein summation convention, we can write
\[
A=d+a_{i}^{j}s_{j}\otimes s^{i}
\]
where $s^{1},...,s^{k}$ is the dual frame of $E^{*}$ and $d$ is
the trivial connection with respect to these frames. Now, denote by
$\tilde{s}_{1},...,\tilde{s}_{k}$ the pulled-back local frame of
$u^{*}E$, and by $\tilde{s}^{1},...,\tilde{s}^{k}$ the dual frame
of $u^{*}E^{*}$. With respect to this frame, we have
\[
u^{*}A=d+\left(u^{*}a_{i}^{j}\right)\tilde{s}_{j}\otimes\tilde{s}^{i}.
\]
By Lemma \ref{lem:pull-back-0-1-forms}, the pull-backs $u^{*}a_{i}^{j}$
are all $0$-$1$-forms. Therefore, $u^{*}A$ is a $0$-connection.
\end{proof}

\subsection{\label{subsec:Harmonic-and-biharmonic simple b-maps}Harmonic and
biharmonic simple $b$-maps}

We now equip $M$ and $N$ with conformally compact metrics $g$ and
$h$, and we discuss the harmonic and biharmonic equation for $u$
with respect to these metrics.

To begin with, let's restrict ourselves to the interiors. The differential
$du:TM^{\circ}\to TN^{\circ}$ can be seen as a section of $T^{*}M^{\circ}\otimes u^{*}TN^{\circ}$
over $M$. Denote by $\nabla$ the connection on $T^{*}M^{\circ}\otimes u^{*}TN^{\circ}$
defined as the tensor product of the Levi-Civita connection $\nabla^{g}$
on $T^{*}M^{\circ}$ and the pull-back of the Levi-Civita connection
$\nabla^{h}$ on $TN^{\circ}$. Then the total derivative $\nabla du$
is a section of $T^{*}M^{\circ}\otimes T^{*}M^{\circ}\otimes u^{*}TN^{\circ}$.
Taking the trace in the first two components, we obtain a section
\[
\tau\left(u\right)=\tr_{g}\nabla du\in C^{\infty}\left(M^{\circ};u^{*}TN^{\circ}\right)
\]
called the \emph{tension field} of $u$.
\begin{defn}
The map $u$ is called \emph{harmonic} if $\tau\left(u\right)=0$.
\end{defn}

The equation $\tau\left(u\right)=0$ is a second order, nonlinear
elliptic equation for $u$. This means that the linearization of the
map $u\mapsto\tau\left(u\right)$ at a solution $v$ is a second order
elliptic operator. More precisely, the linearization is the \emph{Jacobi
operator}
\[
J_{v}=\Delta+\tr_{g}R_{h}\left(dv,\cdot\right)dv
\]
on sections of $v^{*}TN^{\circ}$, where $\Delta=\left(v^{*}\nabla^{h}\right)^{*}v^{*}\nabla^{h}$
is the rough Laplacian. This operator of course is well-defined regardless
of whether $v$ is harmonic or not. This leads us to the notion of
\emph{biharmonic }map:
\begin{defn}
The map $u$ is called \emph{biharmonic} if $J_{u}\left(\tau\left(u\right)\right)=0$.
\end{defn}

Since $g$ and $h$ are metrics only in the interiors, a priori, the
equations $\tau\left(u\right)=0$ and $J_{u}\left(\tau\left(u\right)\right)=0$
make sense only on $M^{\circ}$. However, the fact that $u$ is a
simple $b$-map and $g,h$ are conformally compact, implies that we
can extend the tension field $\tau\left(u\right)$ and the Jacobi
operator $J_{u}$ to the whole of $M$.
\begin{lem}
$ $
\begin{enumerate}
\item The tension field $\tau\left(u\right)$ extends from the interior
of $M$ to a polyhomogeneous section of $u^{*}{^{0}TN}$.
\item The Jacobi operator $J_{u}$ extends from the interior of $M$ to
a second order $0$-elliptic $0$-differential operator with polyhomogeneous
coefficients, acting on sections of $u^{*}{^{0}TN}$.
\end{enumerate}
\end{lem}

\begin{proof}
(1) By Corollary \ref{cor:du-is-0-object}, $du$ extends from $M^{\circ}$
to a section of ${^{0}T^{*}M}\otimes u^{*}{^{0}TN}$ over $M$. Now,
since $g$ is a conformally compact metric on $M$, the Levi-Civita
connection on $T^{*}M^{\circ}$ extends to a $0$-connection on ${^{0}T^{*}M}$;
similarly, the Levi-Civita connection on $TN^{\circ}$ induced by
$h$ extends to a $0$-connection on $^{0}TN$. By Corollary \ref{cor:pull-back-LC-conn-is-0-conn},
the pull-back of the latter via $u$ is a $0$-connection on $u^{*}{^{0}TN}$.
It follows that the tensor product connection $\nabla$ is a $0$-connection
on ${^{0}T^{*}M}\otimes u^{*}{^{0}TN}$. As such, $\nabla\left({^{0}du}\right)$
is a section of ${^{0}T^{*}M}\otimes{^{0}T^{*}M}\otimes u^{*}{^{0}TN}$.
Now, since $g$ is conformally compact, it induces a bundle metric
on $^{0}T^{*}M$; therefore, $\tr_{g}\nabla\left({^{0}du}\right)$
is a section of $u^{*}{^{0}TN}$ as claimed. (2) By Corollary \ref{cor:pull-back-LC-conn-is-0-conn},
$u^{*}\nabla^{h}$ is a $0$-connection on $u^{*}{^{0}TN}$, and therefore
$\Delta:=\left(u^{*}\nabla^{h}\right)^{*}u^{*}\nabla^{h}$ is a second
order $0$-elliptic $0$-differential operator on sections of $u^{*}{^{0}TN}$.
It remains to prove that $\tr_{g}R_{h}\left(du,\cdot\right)du$ extends
to an endomorphism of $u^{*}{^{0}TN}$. Since the Levi-Civita connection
of $h$ is a $0$-connection on $^{0}TN$, the curvature tensor $R_{h}$
extends to a bundle map
\[
R_{h}:{}^{0}TN\otimes{^{0}TN}\otimes{^{0}TN}\to{^{0}TN}.
\]
By Corollary \ref{cor:du-is-0-object}, $du$ extends to a bundle
map ${^{0}du}:{}^{0}TM\to{^{0}TN}$ covering $u$. Therefore, given
$p\in M$, we can interpret $\left(R_{h}\left(^{0}du,\cdot\right){^{0}du}\right)_{p}$
as a linear map ${^{0}T_{p}M}\otimes{^{0}T_{u\left(p\right)}N}\otimes{^{0}T_{p}M}\to{^{0}T_{u\left(p\right)}N}$.
Taking the trace in the $^{0}T_{p}M$ components with respect to $g$,
we obtain that $\left(\tr_{g}R_{h}\left(^{0}du,\cdot\right){^{0}du}\right)_{p}$
is indeed a linear map ${^{0}T_{u\left(p\right)}N}\to{^{0}T_{u\left(p\right)}N}$.
\end{proof}

\subsection{\label{subsec:The--differential-on-the-bdry}Model simple $b$-maps
for $u$}

As mentioned in the introduction, the main goal of this paper is to
prove that, under some hypothesis, a simple $b$-map $u:M\to N$ is
biharmonic if and only if it is harmonic. As we shall see, a crucial
step in this direction is the detailed understanding of a family $u_{p}:M_{p}\to N_{u\left(p\right)}$
of ``model simple $b$-maps'' associated to $u$, parametrized by
a point $p\in\partial M$.

Fix $p\in\partial M$, and define $q=u\left(p\right)$. As explained
in §\ref{sec:0-world}, the inward-pointing closed tangent half-spaces
$M_{p}\subseteq T_{p}M$ and $N_{q}\subseteq T_{q}N$ are naturally
equipped with rescaled hyperbolic metrics $g_{p},h_{q}$ induced by
$g,h$. Since $u$ is a simple $b$-map, the differential $du_{p}:T_{p}M\to T_{q}N$
restricts to a simple $b$-map $u_{p}:M_{p}\to N_{q}$ which should
be thought of as an asymptotic model for $u$ near $p$. Let's describe
$u_{p}$ in coordinates. Choose normal half-space coordinates $\left(x,\boldsymbol{y}\right)$
for $M$ centered at $p$, and normal half-space coordinates $\left(X,\boldsymbol{Y}\right)$
for $N$ centered at $q$, so that in the induced linear coordinates
on $M_{p},N_{q}$ (still denoted by $\left(x,\boldsymbol{y}\right)$
and $\left(X,\boldsymbol{Y}\right)$ with slight abuse of notation)
we have
\[
g_{p}=a_{p}^{-2}\frac{dx^{2}+d\boldsymbol{y}^{2}}{x^{2}},h_{q}=A_{q}^{-2}\frac{dX^{2}+d\boldsymbol{Y}^{2}}{X^{2}}.
\]
Write $u$ in these coordinates as $u=\left(u^{0},\tilde{\boldsymbol{u}}\right)$,
with $\tilde{\boldsymbol{u}}=\left(u^{1},...,u^{n}\right)$. With
respect to the frames $\partial_{x},\partial_{y^{1}},...,\partial_{y^{m}}$
and $\partial_{X},\partial_{Y^{1}},...,\partial_{Y^{n}}$, the differential
$du$ can be written in block matrix form as
\[
du=\left(\begin{matrix}\partial_{x}u^{0} & \partial_{\boldsymbol{y}}u^{0}\\
\partial_{x}\tilde{\boldsymbol{u}} & \partial_{\boldsymbol{y}}\tilde{\boldsymbol{u}}
\end{matrix}\right).
\]
The fact that $u$ is a simple $b$-map however implies that the expression
of $du_{p}$ simplifies. Indeed, $u^{0}$ is a local boundary defining
function for $M$, so $u^{0}=e^{\varphi}x$ for some function $\varphi\left(x,\boldsymbol{y}\right)$:
therefore, $\partial_{y^{i}}u^{0}=\partial_{y^{i}}\left(e^{\varphi}\right)x$
vanishes at $\left(x,\boldsymbol{y}\right)=\left(0,\boldsymbol{0}\right)$.
It follows that the model simple $b$-map $u_{p}:M_{p}\to N_{q}$
in these coordinates takes the simple form
\[
u_{p}\left(x,\boldsymbol{y}\right)=\left(\begin{matrix}\gamma & \boldsymbol{0}\\
\boldsymbol{\beta} & \boldsymbol{\Lambda}
\end{matrix}\right)\left(\begin{matrix}x\\
\boldsymbol{y}
\end{matrix}\right),
\]
where
\begin{align*}
\gamma & =\partial_{x}u^{0}\left(0,\boldsymbol{0}\right)\\
\beta^{i} & =\partial_{x}u^{i}\left(0,\boldsymbol{0}\right)\\
\Lambda_{j}^{i} & =\partial_{y^{j}}u^{i}\left(0,\boldsymbol{0}\right).
\end{align*}
This expression can be simplified further, by an appropriate choice
of the normal half-space coordinates $\left(x,\boldsymbol{y}\right)$
and $\left(X,\boldsymbol{Y}\right)$. Indeed, by a constant rescaling
of either $\left(x,\boldsymbol{y}\right)$ or $\left(X,\boldsymbol{Y}\right)$,
we can assume that $\gamma=1$; moreover, by a constant rotation of
$\boldsymbol{y}$ and $\boldsymbol{Y}$ around the origins, we can
assume that $\boldsymbol{\Lambda}$ is diagonal.

The vector $\boldsymbol{\beta}\in\mathbb{R}^{n}$ and the matrix $\boldsymbol{\Lambda}\in\mathbb{R}^{n\times m}$
are coordinate expressions of invariantly defined objects $\boldsymbol{\beta}_{u},\boldsymbol{\Lambda}_{u}$
naturally associated to $u$. Let $x$ and $X$ be boundary defining
functions for $M$ and $N$, respectively. As observed in §\ref{sec:0-world},
the $0$-$1$-forms $dx/x$ and $dX/X$ are canonical along the boundaries
of $M$ and $N$, respectively. Denote by $x\partial_{x}$ the $g$-dual
of $dx/x$, and call $X\partial_{X}$ the $h$-dual of $dX/X$. These
$0$-vector fields, when restricted to the boundaries, are intrinsic
invariants of $g$ and $h$. We define
\[
\boldsymbol{\beta}_{u}=\left({^{0}du}\left(x\partial_{x}\right)-X\partial_{X}\right)_{|\partial M}.
\]
$\boldsymbol{\beta}_{u}$ is a section of $u^{*}{^{0}TN}$ along $\partial M$,
and it is easy to see that given choices of normal half-space coordinates
$\left(x,\boldsymbol{y}\right)$ and $\left(X,\boldsymbol{Y}\right)$
centered at $p,q$ respectively, we have 
\[
\boldsymbol{\beta}_{u|p}=\sum_{i=1}^{n}\frac{\partial_{x}u^{i}}{e^{\varphi}}\left(0,\boldsymbol{0}\right)X\partial_{Y^{i}}
\]
where $u^{0}=e^{\varphi}x$. Now observe that, since $u$ is a simple
$b$-map, we have $u^{*}\left(dX/X\right)_{|\partial M}=dx/x_{|\partial M}$.
This implies that the $0$-differential $^{0}du_{p}:{^{0}T_{p}M}\to{^{0}T_{q}N}$
maps the kernel of $dx/x_{|p}$ to the kernel of $dX/X_{|q}$. As
$p$ varies, this restriction defines a bundle map
\[
\boldsymbol{\Lambda}_{u}:\ker\frac{dx}{x}_{|\partial M}\to\ker\frac{dX}{X}_{|\partial N}
\]
covering the restriction $u_{|\partial M}:\partial M\to\partial N$.
In the chosen coordinates, we have
\[
\boldsymbol{\Lambda}_{u|p}=\sum_{i=1}^{n}\sum_{j=1}^{m}\frac{\partial_{y^{j}}u^{i}}{e^{\varphi}}\left(0,\boldsymbol{0}\right)X\partial_{Y^{i}}\otimes\frac{dy^{j}}{x}.
\]
The two tensors $\boldsymbol{\beta}_{u}$ and $\boldsymbol{\Lambda}_{u}$
along $\partial M$ characterize the restriction $^{0}du_{|\partial M}$
completely: more precisely, we have
\begin{equation}
^{0}du_{|\partial M}=X\partial_{X}\otimes\frac{dx}{x}+\boldsymbol{\beta}_{u}\otimes\frac{dx}{x}+\boldsymbol{\Lambda}_{u}.\label{eq:0-differential}
\end{equation}

\section{The model problem}

In this section, we study the normal and indicial operators of $J_{u}$.
Let $p\in\partial M$, and define $q=u\left(p\right)$. As explained
in §\ref{subsec:The--differential-on-the-bdry}, the restriction of
the differential $du_{p}$ to the half-tangent spaces $M_{p}$ and
$N_{q}$ defines a model simple $b$-map $u_{p}:M_{p}\to N_{q}$.
The two manifolds $M_{p}$ and $N_{q}$ are equipped with rescaled
hyperbolic metrics $g_{p}$ and $h_{q}$, and therefore it makes sense
to talk about the tension field $\tau\left(u_{p}\right)$ and the
Jacobi operator $J_{u_{p}}$.
\begin{prop}
\label{prop:characterization-normal-operator}The normal operator
of $J_{u}$ at $p$ is the Jacobi operator $J_{u_{p}}$ of the model
simple $b$-map $u_{p}:\left(M_{p},g_{p}\right)\to\left(N_{q},h_{q}\right)$.
\end{prop}

\begin{proof}
It suffices to write $J_{u}$ in normal half-space coordinates $\left(x,\boldsymbol{y}\right)$
for $M$ centered at $p$, and normal half-space coordinates for $\left(X,\boldsymbol{Y}\right)$
for $N$ centered at $u\left(p\right)$, and then evaluate the coefficients
of $J_{u}$ at $\left(x,\boldsymbol{y}\right)=\left(0,\boldsymbol{0}\right)$.
\end{proof}
The previous proposition implies that, to compute the indicial roots
of $J_{u}$, it suffices to compute the indicial roots of model simple
$b$-maps of the form
\begin{align*}
v:\text{H}^{m+1}\left(a\right) & \to\text{H}^{n+1}\left(A\right)\\
\left(\begin{matrix}x\\
\boldsymbol{y}
\end{matrix}\right) & \mapsto\left(\begin{matrix}1 & \boldsymbol{0}\\
\boldsymbol{\beta} & \boldsymbol{\Lambda}
\end{matrix}\right)\left(\begin{matrix}x\\
\boldsymbol{y}
\end{matrix}\right),
\end{align*}
where $\text{H}^{m+1}\left(a\right)$ and $\text{H}^{n+1}\left(A\right)$
are rescaled hyperbolic spaces with coordinates $\left(x,\boldsymbol{y}\right)$,
$\left(X,\boldsymbol{Y}\right)$ and metrics
\[
g=a^{-2}\frac{dx^{2}+d\boldsymbol{y}^{2}}{x^{2}},h=A^{-2}\frac{dX^{2}+d\boldsymbol{Y}^{2}}{X^{2}}
\]
for some $a,A\in\mathbb{R}$, while $\boldsymbol{\beta}\in\mathbb{R}^{n\times1}$
and $\boldsymbol{\Lambda}\in\mathbb{R}^{n\times m}$ has only diagonal
entries. Indeed, if $u$ is a simple $b$-map $M\to N$, by appropriate
choices of normal half-space coordinates for $\left(M,g\right)$ and
$\left(N,h\right)$ centered at $p$ and $u\left(p\right)$ respectively,
we can always write $u_{p}$ as above.

\subsection{The model tension field}

First of all, we compute the tension field of $v$. Fix oriented orthonormal
frames $e_{0},...,e_{m}$ of $^{0}T\text{H}^{m+1}\left(a\right)$
and $E_{0},...,E_{n}$ of $^{0}T\text{H}^{n+1}\left(A\right)$ as
follows:
\begin{align*}
e_{0}=-ax\partial_{x}, & e_{i}=ax\partial_{y^{i}}\\
E_{0}=-AX\partial_{X}, & E_{i}=AX\partial_{Y^{i}},
\end{align*}
and denote by $e^{0},...,e^{m}$ and $E^{0},...,E^{n}$ the dual coframes.
Our orientation convention follows the ``outer-normal-first'' rule:
the induced Stokes orientation on the boundary $\partial\text{H}^{n+1}\left(A\right)=\mathbb{R}^{n}$
is the standard orientation on $\mathbb{R}^{n}$. Note that
\begin{align*}
de^{0}=0, & de^{i}=ae^{0i}\\
dE^{0}=0, & dE^{i}=AE^{0i}
\end{align*}
where $i>0$ and $e^{0i}=e^{0}\land e^{i}$, $E^{0i}=E^{0}\land E^{i}$.
Using these formulas, it is straightforward to compute the Levi-Civita
$0$-connections $\nabla'$ on $^{0}T^{*}\text{H}^{m+1}\left(a\right)$
and $\nabla''$ on $^{0}T\text{H}^{n+1}\left(A\right)$: we have
\begin{align*}
\nabla'e^{0} & =\sum_{i=1}^{m}ae^{i}\otimes e^{i}\\
\nabla'e^{i} & =-ae^{i}\otimes e^{0}
\end{align*}
and, similarly,
\begin{align*}
\nabla''E_{0} & =\sum_{i=1}^{n}AE^{i}\otimes E_{i}\\
\nabla''E_{i} & =-AE^{i}\otimes E_{0}.
\end{align*}
The pull-backs of the $0$-$1$-forms $E^{0},...,E^{n}$ are
\begin{align*}
v^{*}E^{0} & =A^{-1}ae^{0}\\
v^{*}E^{i} & =\begin{cases}
A^{-1}a\left(-\beta^{i}e^{0}+\Lambda_{i}^{i}e^{i}\right) & i\leq m\\
A^{-1}a\left(-\beta^{i}e^{0}\right) & i>m
\end{cases}.
\end{align*}
Therefore, by Formula \ref{eq:0-differential}, the $0$-differential
$^{0}dv$ is
\[
^{0}dv=A^{-1}a\left[e^{0}\otimes E_{0}-\sum_{i=1}^{n}e^{0}\otimes\beta^{i}E_{i}+\sum_{i=1}^{\min\left(m,n\right)}\Lambda_{i}^{i}e^{i}\otimes E_{i}\right].
\]
Given the previous formulas, it is now elementary to compute the tension
field $\tau\left(v\right)=\tr_{g}\nabla\left(^{0}dv\right)$. We omit
the simple algebraic passages, and we report the final result:
\begin{equation}
\tau\left(v\right)=A^{-1}a^{2}\left[\left(m-\boldsymbol{\beta}^{T}\boldsymbol{\beta}-\tr\left(\boldsymbol{\Lambda}^{T}\boldsymbol{\Lambda}\right)\right)E_{0}-\left(m+1\right)\sum_{i=1}^{n}\beta^{i}E_{i}\right].\label{eq:tension}
\end{equation}
From this expression, we obtain the following
\begin{prop}
\label{prop:model-harmonic-simple-bmap}The model simple $b$-map
$v$ is harmonic if and only if $\boldsymbol{\beta}=0$ and $\tr\left(\boldsymbol{\Lambda}^{T}\boldsymbol{\Lambda}\right)=m$.
\end{prop}

\subsection{The model Jacobi operator and its indicial roots}

The model Jacobi operator is
\[
J_{v}=\Delta+\tr_{g}R_{h}\left(^{0}dv,\cdot\right){^{0}dv},
\]
where $\Delta=\left(v^{*}\nabla''\right)^{*}v^{*}\nabla''$ is the
rough Laplacian. The associated indicial operator is
\[
I\left(J_{v}\right)=I\left(v^{*}\nabla''\right)^{*}I\left(v^{*}\nabla''\right)+\tr_{g}R_{h}\left(^{0}dv,\cdot\right){^{0}dv}.
\]

Let's first compute $I\left(v^{*}\nabla''\right)$. We already computed
the connection $\nabla''$ and its pull-back via $v$: it follows
that
\begin{align*}
I\left(v^{*}\nabla''\right) & =\left(\frac{dx}{x}\otimes E_{0}\otimes E^{0}\right)x\partial_{x}\\
 & +\sum_{i=1}^{n}\left(\frac{dx}{x}\otimes E_{i}\otimes E^{i}\right)x\partial_{x}\\
 & +\sum_{i=1}^{n}\beta^{i}\frac{dx}{x}\otimes\left(E_{i}\otimes E^{0}-E_{0}\otimes E^{i}\right)\\
 & +\sum_{i=1}^{\min\left(m,n\right)}\Lambda_{i}^{i}\frac{dy^{i}}{x}\otimes\left(E_{i}\otimes E^{0}-E_{0}\otimes E^{i}\right).
\end{align*}
The formal adjoint of $I\left(u^{*}\nabla\right)$ is now straightforward
to compute, observing that the divergence of $x\partial_{x}$ is $-m$
and therefore $\left(x\partial_{x}\right)^{*}=-x\partial_{x}+m$.
This leads to an explicit description of $I\left(u^{*}\nabla\right)^{*}$,
and finally to a description of $I\left(\Delta\right)$. We omit the
elementary computations, and report the final result in block-matrix
form with respect to the frame $E_{0},...,E_{n}$: writing
\[
\sigma^{2}=\boldsymbol{\beta}^{T}\boldsymbol{\beta}+\tr\left(\boldsymbol{\Lambda}^{T}\boldsymbol{\Lambda}\right),
\]
we have
\[
I\left(\Delta\right)=a^{2}\left[-\boldsymbol{I}_{n+1}\left(x\partial_{x}\right)^{2}+\left(\begin{matrix}m & -2\boldsymbol{\beta}^{T}\\
2\boldsymbol{\beta} & m\boldsymbol{I}_{n}
\end{matrix}\right)x\partial_{x}+\left(\begin{matrix}\sigma^{2} & m\boldsymbol{\beta}^{T}\\
-m\boldsymbol{\beta} & \boldsymbol{\beta}\boldsymbol{\beta}^{T}+\boldsymbol{\Lambda}\boldsymbol{\Lambda}^{T}
\end{matrix}\right)\right].
\]

Let's now compute the curvature term $\tr_{g}R_{h}\left(^{0}dv,\cdot\right){^{0}dv}$.
Directly from the expression of the Levi-Civita $0$-connection $\nabla$
of $h$, we get
\begin{align*}
R_{h} & =A^{2}\left[\sum_{i=1}^{n}E^{0i}\otimes\left(E_{i}\otimes E^{0}-E_{0}\otimes E^{i}\right)-\sum_{\begin{smallmatrix}i,j=1\\
i\not=j
\end{smallmatrix}}^{n}E^{ij}\otimes\left(E_{i}\otimes E^{j}\right)\right].
\end{align*}
Now, we already computed $^{0}dv$ above, so it is easy to compute
$\tr_{g}R_{h}\left(^{0}dv,\cdot\right){}^{0}dv$. Omitting the various
elementary steps, the final result is
\[
\tr_{g}R_{h}\left(^{0}dv,\cdot\right){}^{0}dv=a^{2}\left[\sigma^{2}\boldsymbol{I}_{n+1}+\left(\begin{matrix}0 & \boldsymbol{\beta}^{T}\\
\boldsymbol{\beta} & \boldsymbol{I}_{n}
\end{matrix}\right)-\left(\begin{matrix}0 & \boldsymbol{0}\\
\boldsymbol{0} & \boldsymbol{\beta}\boldsymbol{\beta}^{T}+\boldsymbol{\Lambda}\boldsymbol{\Lambda}^{T}
\end{matrix}\right)\right].
\]
Adding this term to $I\left(\Delta\right)$, we finally obtain the
indicial operator of $J_{v}$:
\[
I\left(J_{v}\right)=a^{2}\left[-\boldsymbol{I}_{n+1}\left(x\partial_{x}\right)^{2}+\left(\begin{matrix}m & -2\boldsymbol{\beta}^{T}\\
2\boldsymbol{\beta} & m\boldsymbol{I}_{n}
\end{matrix}\right)\left(x\partial_{x}\right)+\left(\begin{matrix}2\sigma^{2} & \left(m+1\right)\boldsymbol{\beta}^{T}\\
\left(-m+1\right)\boldsymbol{\beta} & \left(1+\sigma^{2}\right)\boldsymbol{I}_{n}
\end{matrix}\right)\right].
\]
The indicial roots of $J_{v}$ are therefore the values of $z$ for
which the matrix
\[
-\boldsymbol{I}_{n+1}z^{2}+\left(\begin{matrix}m & -2\boldsymbol{\beta}^{T}\\
2\boldsymbol{\beta} & m\cdot\boldsymbol{I}_{n}
\end{matrix}\right)z+\left(\begin{matrix}2\sigma^{2} & \left(m+1\right)\boldsymbol{\beta}^{T}\\
\left(-m+1\right)\boldsymbol{\beta} & \left(1+\sigma^{2}\right)\boldsymbol{I}_{n}
\end{matrix}\right)
\]
is singular.
\begin{lem}
The model Jacobi operator $J_{v}$ has $0$ as an indicial root if
and only if $\boldsymbol{\beta}=0$ and $\boldsymbol{\Lambda}=0$.
\end{lem}

\begin{proof}
$0$ is an indicial root of $J_{v}$ if and only if the determinant
\[
\det\left(\begin{matrix}2\sigma^{2} & \left(m+1\right)\boldsymbol{\beta}^{T}\\
\left(-m+1\right)\boldsymbol{\beta} & \left(1+\sigma^{2}\right)\boldsymbol{I}_{n}
\end{matrix}\right)=\left(2\sigma^{2}+\frac{\left(m+1\right)\left(m-1\right)\boldsymbol{\beta}^{T}\boldsymbol{\beta}}{1+\sigma^{2}}\right)\left(1+\sigma^{2}\right)^{n}
\]
is zero. Canceling the $1+\sigma^{2}$ factors and substituting $\sigma^{2}=\boldsymbol{\beta}^{T}\boldsymbol{\beta}+\tr\left(\boldsymbol{\Lambda}^{T}\boldsymbol{\Lambda}\right)$,
this condition is equivalent to the equation
\[
\sigma^{2}+2\sigma^{4}+\tr\left(\boldsymbol{\Lambda}^{T}\boldsymbol{\Lambda}\right)+m^{2}\boldsymbol{\beta}^{T}\boldsymbol{\beta}=0,
\]
which is equivalent to the vanishing of $\boldsymbol{\beta}$ and
$\boldsymbol{\Lambda}$.
\end{proof}
From this lemma, we obtain immediately the model version of our main
theorem:
\begin{prop}
\label{prop:model-biharmonic-->harmonic}Let $v$ be biharmonic. If
$m\geq1,$ assume that $\boldsymbol{\Lambda}\not=0$. Then $v$ is
harmonic, and consequently $\boldsymbol{\beta}=0$ and $\tr\left(\boldsymbol{\Lambda}^{T}\boldsymbol{\Lambda}\right)=m$
by Proposition \ref{prop:model-harmonic-simple-bmap}.
\end{prop}

\begin{proof}
In general, for a $0$-differential operator $L\in\Diff_{0}^{\bullet}\left(M;E,F\right)$
between sections of vector bundles $E,F\to M$, the indicial polynomial
$I_{p}\left(L\right)\left(z\right)$ at $p\in\partial M$ can also
be expressed as the polynomial $z\mapsto\left(x^{-z}Lx^{z}\right)_{p}$
with coefficients in the space of linear maps $E_{p}\to F_{p}$, where
$x$ is a boundary defining function for $M$. This implies that $I_{p}\left(L\right)\left(0\right)$
has non-trivial kernel if and only if there exists a section $u$
of $E$, with $u_{p}\not=0$, such that $\left(Lu\right)_{p}=0$.
We now apply this observation to our problem. Assume first $m=0$,
so that 
\[
v\left(x\right)=\left(x,\beta^{1}x,...,\beta^{n}x\right),\tau\left(v\right)=A^{-1}a^{2}\left[-\boldsymbol{\beta}^{T}\boldsymbol{\beta}E_{0}-\sum_{i=1}^{n}\beta^{i}E_{i}\right].
\]
Assume that $v$ is biharmonic, and suppose for the sake of a contradiction
that $v$ is not harmonic. Then $\tau\left(v\right)\not=0$, and $J_{v}\tau\left(v\right)=0$.
However, this implies that $0$ is an indicial root of $J_{v}$, and
therefore $\boldsymbol{\beta}=0$ by the previous lemma, contradicting
$\tau\left(v\right)\not=0$. Now, consider the case $m\geq1$ and
$\boldsymbol{\Lambda}\not=0$. Again, if $v$ is biharmonic not harmonic,
$0$ must be an indicial root of $J_{v}$, which implies that $\boldsymbol{\beta}=0$
and $\boldsymbol{\Lambda}=0$, contradicting the hypothesis $\boldsymbol{\Lambda}\not=0$.
\end{proof}
From these computations, we also obtain an example of biharmonic,
not harmonic, model simple $b$-map.
\begin{prop}
\label{prop:example-proper-biharmonic}Let $m\geq1$. Then the model
simple $b$-map
\begin{align*}
v:\text{H}^{m+1}\left(a\right) & \to\text{H}^{n+1}\left(A\right)\\
\left(x,\boldsymbol{y}\right) & \mapsto\left(x,\boldsymbol{0}\right)
\end{align*}
is biharmonic but not harmonic.
\end{prop}

\begin{proof}
This case corresponds to $\boldsymbol{\beta}=0$ and $\boldsymbol{\Lambda}=0$.
By Formula \ref{eq:tension}, the tension field is
\[
\tau\left(v\right)=-a^{2}mX\partial_{X},
\]
so since $m\geq1$, $v$ is not harmonic (for $m=0$, $v$ is in fact
totally geodesic). However, $J_{v}$ has an indicial root $0$. More
precisely, the indicial polynomial of $J_{v}$ evaluated at $z=0$
is
\[
\left(\begin{matrix}0 & \boldsymbol{0}\\
\boldsymbol{0} & \boldsymbol{I}_{n}
\end{matrix}\right).
\]
This fact, together with the fact that $\tau\left(v\right)$ is a
constant multiple of $X\partial_{X}$, implies that $J_{v}\left(\tau\left(v\right)\right)=0$,
so $v$ is biharmonic as claimed.
\end{proof}
Assuming that $v$ is harmonic, it is now very easy to compute all
the indicial roots:
\begin{prop}
\label{prop:ind-roots}Let $v$ be harmonic. Then the indicial roots
of $J_{v}$ are real and come in two pairs $\alpha_{\pm},\beta_{\pm}$
symmetric about $m/2$:
\begin{align*}
\alpha_{\pm} & =\frac{m}{2}\pm\frac{\sqrt{m^{2}+8m}}{2}\\
\beta_{-} & =-1\\
\beta_{+} & =m+1.
\end{align*}
In particular, if $m\geq1$, there are no indicial roots in the interval
$\left(-1,m+1\right)$.
\end{prop}

\begin{proof}
Since $v$ is harmonic, we have $\boldsymbol{\beta}=0$ and $\tr\left(\boldsymbol{\Lambda}^{T}\boldsymbol{\Lambda}\right)=m$.
The indicial polynomial in this case simplifies to
\[
\left(\begin{matrix}-z^{2}+mz+2m\\
 & -z^{2}+mz+\left(1+m\right)\boldsymbol{I}_{n}
\end{matrix}\right).
\]
It follows that the indicial roots are the solution $\alpha_{\pm}$
of $-z^{2}+mz+2m=0$, and $\beta_{\pm}=-z^{2}+mz+\left(1+m\right)=0$,
which are reported above.
\end{proof}

\section{Proof of the main theorem}

Let's come back to the global picture.
\begin{prop}
Assume $m\geq1$. Let $u:M\to N$ be a biharmonic simple $b$-map,
and assume that $u$ is not constant on any connected component of
$\partial M$. Then $u$ is asymptotically harmonic, i.e. the tension
field $\tau\left(u\right)\in C^{\infty}\left(M;u^{*}{^{0}TN}\right)$
vanishes along $\partial M$; equivalently, $\boldsymbol{\beta}_{u}=0$
and $\tr\left(\boldsymbol{\Lambda}_{u}^{*}\boldsymbol{\Lambda}_{u}\right)\equiv m$.
Moreover, the indicial roots of the Jacobi operator $J_{u}$ are $\alpha_{\pm},\beta_{\pm}$
as in Proposition \ref{prop:ind-roots}.
\end{prop}

\begin{proof}
Since $u$ is not constant on any connected component of $\partial M$,
for every connected component $C$ of $\partial M$ there is a $p\in C$
such that the differential $d\left(u_{|\partial M}\right)_{p}$ is
non-zero. This condition is equivalent to $\boldsymbol{\Lambda}_{u|p}\not=0$,
and therefore $\tr\left(\boldsymbol{\Lambda}_{u}^{*}\boldsymbol{\Lambda}_{u}\right)\left(p\right)\not=0$.
Call now $q=u\left(p\right)$. Writing the equation $J_{u}\left(\tau\left(u\right)\right)=0$
in normal half-space coordinates centered at $p$ and $q$, we obtain
that we must have $J_{u_{p}}\left(\tau\left(u_{p}\right)\right)=0$
for the model simple $b$-map $u_{p}:M_{p}\to N_{q}$. By Proposition
\ref{prop:model-biharmonic-->harmonic}, the fact that $\boldsymbol{\Lambda}_{u|p}\not=0$
implies that $u_{p}$ is actually harmonic, i.e. $\tau\left(u_{p}\right)=\tau\left(u\right)_{p}=0$.
Moreover, again by Proposition \ref{prop:model-biharmonic-->harmonic},
we have $\boldsymbol{\beta}_{u|p}=0$ and $\tr\left(\boldsymbol{\Lambda}_{u}^{*}\boldsymbol{\Lambda}_{u}\right)\left(p\right)=m$.
This argument shows that the function $\tr\left(\boldsymbol{\Lambda}_{u}^{*}\boldsymbol{\Lambda}_{u}\right)\in C^{\infty}\left(\partial M\right)$
takes values in the discrete set $\left\{ 0,m\right\} $. Therefore,
since $\boldsymbol{\Lambda}_{u}$ is not identically zero on every
connected component of $\partial M$, $\tr\left(\boldsymbol{\Lambda}_{u}^{*}\boldsymbol{\Lambda}_{u}\right)$
must be equal to $m$ on the whole of $\partial M$, and $\boldsymbol{\beta}_{u}$
must vanish everywhere on $\partial M$. The statement about the indicial
roots follows from the fact that the indicial roots of $J_{u}$ at
$p$ coincide with the indicial roots of $J_{u_{p}}$.
\end{proof}
We can finally prove our main theorem.
\begin{thm}
Let $\left(M^{m+1},g\right)$ and $\left(N^{n+1},h\right)$ be conformally
compact manifolds, with $\left(N,h\right)$ of non-positive sectional
curvature. Let $u:M\to N$ be a simple $b$-map. If $m\geq1$, assume
that $u_{|\partial M}$ is not constant on any connected component
of $\partial M$. Then $u$ is biharmonic if and only if it is harmonic.
\end{thm}

\begin{proof}
If $m=0$, then the claim follows from Corollary 2.5 of \cite{CaddeoMontaldoOniciucSpheres}.
Now, assume that $m\geq1$. By the previous proposition, $u$ is asymptotically
harmonic, and therefore the indicial roots of $J_{u}$ are the $\alpha_{\pm},\beta_{\pm}$
of Proposition \ref{prop:ind-roots}. In particular, since $m\ge1$,
the interval $\left(-1,m+1\right)$ does not contain any indicial
roots for $J_{u}$. Let now $p\in\partial M$, and call $q=u\left(p\right)$.
By Proposition \ref{prop:characterization-normal-operator}, the normal
operator $N_{p}\left(J_{u}\right)$ equals $J_{u_{p}}$, the Jacobi
operator of the model simple $b$-map $u_{p}:M_{p}\to N_{q}$. This
operator is
\[
J_{u_{p}}=\Delta+\tr_{g}R_{h_{q}}\left(^{0}du_{p},\cdot\right){^{0}du_{p}},
\]
where $\Delta$ is the rough Laplacian on $u_{p}^{*}{^{0}TN_{q}}$.
Now, since the metric $h_{q}$ on $N_{q}$ is a rescaling of the hyperbolic
metric, the term $\tr_{g}R_{h_{q}}\left(^{0}du_{p},\cdot\right){^{0}du_{p}}$
is positive semi-definite: indeed, if $v\in{^{0}T_{p}M}$ and $w\in{^{0}T_{q}N}$,
the quantity
\[
\left\langle R_{h_{q}}\left(^{0}du_{p}\left(v\right),w\right){^{0}du_{p}}\left(v\right),w\right\rangle 
\]
is either zero if ${^{0}du_{p}}\left(v\right),w$ are colinear, or
a negative rescaling of the sectional curvature of $h_{q}$. This
implies that $J_{u_{p}}$ has no $L_{0}^{2}$ kernel: indeed, if $\omega$
is in the $L_{0}^{2}$ kernel of $J_{u_{p}}$, then integrating the
equation $\left(J_{u_{p}}\omega,\omega\right)_{L_{0}^{2}}=0$ by parts
we obtain that $u_{p}^{*}\nabla''\omega=0$, where $\nabla''$ is
the Levi-Civita $0$-connection on $^{0}T\left(N_{q}\right)$. This
implies that $\left|\omega\right|$ is constant on $M_{p}^{\circ}$,
while the $L_{0}^{2}$ condition implies that $\left|\omega\right|$
decays at infinity; therefore, we must have $\omega=0$. We thus proved
that, for every $p\in\partial M$, the normal operator $N_{p}\left(J_{u}\right)$
has vanishing $L^{2}$ kernel. It follows then from Theorem \ref{thm:(Mazzeo)}
that the operator $J_{u}$ is Fredholm of index zero as a map
\[
J_{u}:x^{\delta-\frac{m}{2}}H_{0}^{2}\left(M;u^{*}{^{0}TN}\right)\to x^{\delta-\frac{m}{2}}L_{0}^{2}\left(M;u^{*}{^{0}TN}\right)
\]
for every $\delta\in\left(-1,m+1\right)$, and that the $x^{\delta-\frac{m}{2}}L_{0}^{2}$
kernels are all equal for $\delta$ in this interval. Now, since $u$
is a simple $b$-map, the tension field $\tau\left(u\right)$ is a
section of $u^{*}{^{0}TN}$. Since $u$ is biharmonic, $\tau\left(u\right)$
is in the $x^{\delta-\frac{m}{2}}L_{0}^{2}$ kernel of $J_{u}$ for
every $\delta<0$ (in fact, since $u$ is asymptotically harmonic,
we have $\tau\left(u\right)=O\left(x\right)$ and therefore $\tau\left(u\right)$
is in the $x^{\delta-\frac{m}{2}}L_{0}^{2}$ kernel of $J_{u}$ for
every $\delta<1$). But then, we can integrate the equation $\left(J_{u}\left(\tau\left(u\right)\right),\tau\left(u\right)\right)_{L_{0}^{2}}=0$
by parts, as we did above in the model case. Again, from the fact
that $h$ has non-positive sectional curvature and $\tau\left(u\right)$
decays at infinity, we deduce that $\tau\left(u\right)=0$ as claimed.
\end{proof}
Recall that harmonic isometric immersions must be minimal. Then, directly
from the previous result, we obtain the following version of the Generalized
Chen's Conjecture:
\begin{cor}
Let $\left(N,h\right)$ be a conformally compact manifold of non-positive
sectional curvature, and let $M$ be a compact manifold with boundary.
Let $u:M\to N$ be an immersion mapping $\partial M$ to $\partial N$,
$M^{\circ}$ to $N^{\circ}$, and transversal to $\partial N$. Then
$u$ is biharmonic if and only if it is minimal.
\end{cor}

\bibliography{allpapers}{}

\begin{thebibliography}{CDLS05}

\bibitem[Biq05]{BiquardAdS}
O.~Biquard, editor.
\newblock {\em {AdS/CFT correspondence: Einstein metrics and their conformal
  boundaries. Proceedings, 73rd Meeting of Theoretical Physicists and
  Mathematicians, Strasbourg, France, September 11-13, 2003}}, Zurich,
  Switzerland, 2005. Eur. Math. Soc.

\bibitem[BMO08]{BalmusMontaldoOniciucSpheres}
A.~Balmus, S.~Montaldo, and C.~Oniciuc.
\newblock Classification results for biharmonic submanifolds in spheres.
\newblock {\em Isr. J. Math.}, 168:201--220, 2008.

\bibitem[CDLS05]{BoundaryRegularityPE}
Piotr~T. Chru{\'s}ciel, Erwann Delay, John~M. Lee, and Dale~N. Skinner.
\newblock Boundary regularity of conformally compact {Einstein} metrics.
\newblock {\em J. Differ. Geom.}, 69(1):111--136, 2005.

\bibitem[Che91]{ChenI}
Bang-Yen Chen.
\newblock Some open problems and conjectures on submanifolds of finite type.
\newblock {\em Soochow J. Math.}, 17(2):169--188, 1991.

\bibitem[CMO02]{CaddeoMontaldoOniciucSpheres}
R.~Caddeo, S.~Montaldo, and C.~Oniciuc.
\newblock Biharmonic immersions into spheres.
\newblock In {\em Differential geometry, {V}alencia, 2001}, pages 97--105.
  World Sci. Publ., River Edge, NJ, 2002.

\bibitem[CMP01]{CaddeoMontaldoPiuOnBiharmonicMaps}
Renzo Caddeo, Stefano Montaldo, and Paola Piu.
\newblock On biharmonic maps.
\newblock In {\em Global differential geometry: the mathematical legacy of
  Alfred Gray. Proceedings of the international congress on differential
  geometry held in memory of Professor Alfred Gray, Bilbao, Spain, September
  18--23, 2000}, pages 286--290. Providence, RI: American Mathematical Society
  (AMS), 2001.

\bibitem[EL83]{EellsLemaireSelectedTopics}
James Eells and Luc Lemaire.
\newblock {\em Selected topics in harmonic maps}, volume~50 of {\em Reg. Conf.
  Ser. Math.}
\newblock Providence, RI: American Mathematical Society (AMS), 1983.

\bibitem[ES64]{EellsSampson}
James~jun. Eells and J.~H. Sampson.
\newblock Harmonic mappings of {Riemannian} manifolds.
\newblock {\em Am. J. Math.}, 86:109--160, 1964.

\bibitem[FG12]{FeffermanGraham}
Charles Fefferman and C.~Robin Graham.
\newblock {\em The ambient metric}, volume 178 of {\em Annals of Mathematics
  Studies}.
\newblock Princeton University Press, Princeton, NJ, 2012.

\bibitem[GS21]{SiffertNewBiharmonic}
Sigmundur Gudmundsson and Anna Siffert.
\newblock New biharmonic functions on the compact {Lie} groups
  {{\(\mathbf{SO}(n)\)}}, {{\(\mathbf{SU}(n)\)}}, {{\(\mathbf{Sp}(n)\)}}.
\newblock {\em J. Geom. Anal.}, 31(1):250--281, 2021.

\bibitem[GSS22]{SiffertExplicitProperPharmonic}
Sigmundur Gudmundsson, Anna Siffert, and Marko Sobak.
\newblock Explicit proper {{\(p\)}}-harmonic functions on the riemannian
  symmetric spaces {{\(\mathbf{SU} (n)/\mathbf{SO} (n)\)}}, {{\(\mathbf{Sp}
  (n)/\mathbf{U} (n)\)}}, {{\(\mathbf{SO} (2n)/\mathbf{U} (n)\)}},
  {{\(\mathbf{SU} (2n)/\mathbf{Sp} (n)\)}}.
\newblock {\em J. Geom. Anal.}, 32(5):16, 2022.
\newblock Id/No 147.

\bibitem[Har67]{HartmanHomotopicHarmonic}
P.~Hartman.
\newblock On homotopic harmonic maps.
\newblock {\em Can. J. Math.}, 19:673--687, 1967.

\bibitem[Hin23]{Hintz0calculus}
Peter Hintz.
\newblock Elliptic parametrices in the 0-calculus of {M}azzeo and {M}elrose.
\newblock {\em Pure and Applied Analysis}, 5(3):729--766, 2023.

\bibitem[Jia86a]{JiangII}
Guo~Ying Jiang.
\newblock 2-harmonic maps and their first and second variational formulas.
\newblock {\em Chinese Ann. Math. Ser A}, 7:389--402, 1986.

\bibitem[Jia86b]{JiangI}
Guoying Jiang.
\newblock 2-harmonic isometric immersions between {Riemannian} manifolds.
\newblock {\em Chin. Ann. Math., Ser. A}, 7:130--144, 1986.

\bibitem[Jos08]{JostBook}
J{\"u}rgen Jost.
\newblock Harmonic mappings.
\newblock In {\em Handbook of geometric analysis. No. 1}, pages 147--194.
  Somerville, MA: International Press; Beijing: Higher Education Press, 2008.

\bibitem[Lee06]{LeeFredholm}
John~M. Lee.
\newblock Fredholm operators and {E}instein metrics on conformally compact
  manifolds.
\newblock {\em Mem. Amer. Math. Soc.}, 183(864):vi+83, 2006.

\bibitem[Mae14]{Maeta}
Shun Maeta.
\newblock Biharmonic maps from a complete {Riemannian} manifold into a
  non-positively curved manifold.
\newblock {\em Ann. Global Anal. Geom.}, 46(1):75--85, 2014.

\bibitem[Maz86]{MazzeoPhD}
Rafe~R. Mazzeo.
\newblock {\em {Hodge cohomology of negatively curved manifolds}}.
\newblock ProQuest LLC, Ann Arbor, MI, 1986.
\newblock Thesis (Ph.D.)--Massachusetts Institute of Technology.

\bibitem[Maz88]{MazzeoHodge}
Rafe~R. Mazzeo.
\newblock The {H}odge cohomology of a conformally compact metric.
\newblock {\em J. Differential Geom.}, 28(2):309--339, 1988.

\bibitem[Maz91a]{MazzeoYamabeRegularity}
Rafe Mazzeo.
\newblock Regularity for the singular {Yamabe} problem.
\newblock {\em Indiana Univ. Math. J.}, 40(4):1277--1299, 1991.

\bibitem[Maz91b]{MazzeoEdgeI}
Rafe~R. Mazzeo.
\newblock Elliptic theory of differential edge operators. {I}.
\newblock {\em Comm. Partial Differential Equations}, 16(10):1615--1664, 1991.

\bibitem[Mel96]{MelroseCorners}
Richard~B. Melrose.
\newblock Differential analysis on manifolds with corners, 1996.

\bibitem[MM87]{MazzeoMelroseResolvent}
Rafe~R. Mazzeo and Richard~B. Melrose.
\newblock Meromorphic extension of the resolvent on complete spaces with
  asymptotically constant negative curvature.
\newblock {\em J. Funct. Anal.}, 75(2):260--310, 1987.

\bibitem[MO06]{MontaldoOniciucSurvey}
S.~Montaldo and C.~Oniciuc.
\newblock A short survey on biharmonic maps between {Riemannian} manifolds.
\newblock {\em Rev. Uni{\'o}n Mat. Argent.}, 47(2):1--22, 2006.

\bibitem[MOR15]{MontaldoOniciucRattoRotationally}
S.~Montaldo, C.~Oniciuc, and A.~Ratto.
\newblock Rotationally symmetric biharmonic maps between models.
\newblock {\em J. Math. Anal. Appl.}, 431(1):494--508, 2015.

\bibitem[NUG14]{NUG}
Nobumitsu Nakauchi, Hajime Urakawa, and Sigmundur Gudmundsson.
\newblock Biharmonic maps into a {Riemannian} manifold of non-positive
  curvature.
\newblock {\em Geom. Dedicata}, 169:263--272, 2014.

\bibitem[OC20]{YeLinBangYenChen}
Ye-Lin Ou and Bang-Yen Chen.
\newblock {\em Biharmonic submanifolds and biharmonic maps in {R}iemannian
  geometry}.
\newblock World Scientific Publishing Co. Pte. Ltd., Hackensack, NJ, [2020]
  \copyright 2020.

\bibitem[OT12]{YeLinLiangChenConjecture}
Ye-Lin Ou and Liang Tang.
\newblock On the generalized {Chen}'s conjecture on biharmonic submanifolds.
\newblock {\em Mich. Math. J.}, 61(3):531--542, 2012.

\bibitem[Usu21]{UsulaYM}
Marco Usula.
\newblock Yang-{M}ills connections on conformally compact manifolds.
\newblock {\em Lett. Math. Phys.}, 111(2):Paper No. 56, 23, 2021.

\bibitem[Usu22]{UsulaPhD}
Marco Usula.
\newblock {\em 0-elliptic boundary value problems}.
\newblock 2022.
\newblock Thesis (Ph.D.) KU Leuven, Faculty of Science.

\bibitem[Usu24]{Usula0BVP}
Marco Usula.
\newblock Boundary value problems for 0-elliptic operators.
\newblock {\em arXiv preprint arXiv:2412.06084}, 2024.

\end{thebibliography}
\bibliographystyle{alpha}
\newpage
\end{document}